\documentclass[11pt]{amsart}
\usepackage{amsmath,amssymb}
\usepackage{color}
\usepackage{mathrsfs}
\usepackage{graphicx}
\usepackage{lineno}

\def\phi{\varphi}
\def\H{H^\infty}
\numberwithin{equation}{section}
\theoremstyle{plain}
\newtheorem{theorem}{Theorem}[section]
\newtheorem{lemma}[theorem]{Lemma}

\newtheorem{corollary}[theorem]{Corollary}
\newtheorem{proposition}[theorem]{Proposition}

\newcommand{\reff}[1]{(\ref{#1})}

\def\K{\mathbb K}
\def\T{ \mathbb T}
\def\R{ \mathbb R}
\def\H{H^\infty}

\def\D{{ \mathbb D}}
\def\C{{ \mathbb C}}

\def\N{{ \mathbb N}}
\def\B{{ \mathscr B}}

\def\e{\varepsilon}

\def\union{\cup}

\def\inter{\cap}

\def\ov{\overline}

\def\ss{\subseteq}
\def\emp{\emptyset}

\def\buildrel#1_#2^#3{\mathrel{\mathop{\kern 0pt#1}\limits_{#2}^{#3}}}

\def\BP{Blaschke product}

\def\IBP{interpolating Blaschke product}

\def\imp{\Longrightarrow}

\begin{document}

\title[Stable Ranks for the Sarason Algebra]{Spectral Characteristics and Stable Ranks for the Sarason Algebra $\H+C$}

\author[R. Mortini]{Raymond Mortini}
\address{\small D\'{e}partement de Math\'{e}matiques\\
\small LMAM,  UMR 7122,
\small Universit\'{e} Paul Verlaine\\
\small Ile du Saulcy\\
\small F-57045 Metz, France}

\email{mortini@poncelet.univ-metz.fr}
 
\author[B. D. Wick]{Brett D. Wick$^*$}
\address{Department of Mathematics\\ University of South Carolina\\ LeConte College\\1523 Greene Street\\ Columbia, SC USA 29206}
\email{wick@math.sc.edu}
\thanks{$*$ Research supported in part by a National Science Foundation DMS Grant \# 0752703.}


 \maketitle


   \section*{Introduction}

We prove a Corona type theorem with bounds for the Sarason algebra ${\H+C}$ and determine
its spectral characteristics, thus continuing a line of research initiated by 
N.~Nikolski. We also determine the Bass, the  dense, and the topological stable ranks of $\H+C$.

To fix our setting, let $A$ be a commutative unital Banach algebra with unit $e$ and
$M(A)$ its maximal ideal space.  The following  concept of spectral characteristics was introduced by N. Nikolski \cite{ni}.  For $a \in A$, let $\hat{a}$
denote the Gelfand transform of $a$. We let
  \begin{linenomath}
  $$\delta(a) = \min_{t \in M(A)} |\hat{a}(t)|.$$
  \end{linenomath}
Note that $\delta(a) \le \|\hat{a}\|_\infty \le  \|a\|_A$. When $a
= (a_1, \ldots, a_n) \in A^n$
we define 
  \begin{linenomath}
  $$\delta_n(a) = \min_{t \in M(A)} |\hat{a}(t)|,$$
    \end{linenomath}
     where
$|\hat{a}(t)|=\sum_{j = 1}^n |\hat{a}_j(t)|$ for $t\in M(A)$ and we let 
\begin{linenomath}
  
$$||a||_{A^n}=\max\{||a_1||_A,\dots,||a_n||_A\}.$$
    \end{linenomath}
    Typically, one defines $|\hat{a}(t)| =|\hat{a}(t)|_2:= (\sum_{j = 1}^n |\hat{a}_j(t)|^2)^{1/2}$
    and  $\|a\|_{A^n}=\|a\|_2:=\bigl(\sum_{j = 1}^n \|a_j\|^2_A\bigr)^{1/2}$.
    Our later calculations 
     will be easier though with the present definition.

Let $\delta$ be a real number satisfying $0 < \delta \le 1$.  We
are interested in finding, or bounding, the functions

  $$c_1(\delta,A) = \sup\{\|a^{-1}\|_A:  \|a\|_A \le 1,\;  \delta(a) \geq \delta\}$$ and
  \begin{equation}\label{niko}
  c_n(\delta,A) = \sup\biggl\{\inf\{\|b\|_{A^n}: \sum_{j = 1}^n a_j b_j = e\},\|a\|_{A^n} \le 1,\;  \delta_n(a) \geq \delta \biggr\},
    \end{equation}
when $A$ is the Sarason algebra $\H+C$.
If $a$ is not invertible, we define $\|a^{-1}\| =
\infty$. 
 
It should be
clear that $1\leq c_n(\delta,A) \le c_{n + 1}(\delta,A)$
and, if $0 < \delta^\prime \le \delta \le 1$, then
$c_n(\delta,A) \le c_n(\delta^\prime,A)$. This
implies the existence of a {\it critical constant}, denoted here
by $\delta_n(A)$ such that
  \begin{linenomath}
  $$c_n(\delta,A) = \infty ~\mbox{for}~ 0 < \delta < \delta_n(A)
~\mbox{and}~ c_n( \delta,A) < \infty ~\mbox{for}~
\delta_n(A) < \delta \le 1.$$
  \end{linenomath}

It is clear that if $A$ is a uniform algebra, then $\delta_1(A) = 0$ and $c_1(\delta, A)=1/\delta$.
It is not known for which uniform algebras $\delta_n(A)>0$.

If $A=\H$, the  algebra  of bounded holomorphic functions in the unit disk $\D$,
then the famous Carleson  Corona Theorem tells us that $\delta_n(\H)=0$ for each $n$.
Estimates for $c_n(\delta,\H)$ were given, among others,  by Nikolski \cite{nik}, Rosenblum \cite{ro},  and Tolokonnikov
\cite{to}.
The best known estimate today seems to appear in \cite{tr2}  
and \cite{tw}, see also \cite{trent}.  Here, for the lower bound, $\delta$ is close to 0, and $\kappa$ is a universal constant:

\begin{linenomath}
$$\kappa\;\delta^{-2}\log\log(1/\delta)\leq c^{(2)}_n(\delta,\H)\leq
\frac{1}{\delta} +\frac{17}{\delta^2}\log\frac{1}{\delta},$$
\end{linenomath}
 where $c^{(2)}_n(\delta, A)$
denotes the spectral characteristic $c_n(\delta, A)$ above (see \reff{niko}), whenever
defined with the euclidian norms $|\cdot|_2$ and $\|\cdot\|_2$.
\medskip

The structure of the paper is as follows.  In Section \ref{one} we consider the 
problem of solving Bezout equations in
 the Sarason algebra $\H+C$.
 Indeed, we prove  that the Corona Theorem with bounds holds in $\H+C$ 
 (that is $\delta_n(\H+C)=0$ for each $n$).  We also give explicit estimates of the associated spectral characteristics.

In Section \ref{two} we present different notions of stable ranks that are relevant
to the topic of this paper.  We also show the relationships between these various notions of stable ranks.

 In Section \ref{three} we determine the Bass and  topological stable ranks
 of $\H+C$.  These results can be considered as a generalization of the Corona Theorem for
$\H$. In particular, we investigate whether any Bezout equation $af+bg=1$ in $\H+C$
 admits a solution where $a$ itself is invertible. We also show that on $\H+C$-convex sets
 in $M(\H+C)$ zero free functions admit $\H+C$- invertible approximants. This
 will be used to determine the dense stable rank of $\H+C$.
 

\section{Spectral Characteristics for the Sarason Algebra $\H+C$}
\label{one}

Since we are dealing only with uniform algebras $A$, we shall identify the elements in 
$A$ with their Gelfand transform $\hat f$.

Let  $L^\infty(\T)$ denote the algebra of essentially bounded, Lebesgue measurable functions on the unit circle $\T$.  Then a \textit{Douglas algebra} is a uniformly closed subalgebra of $L^\infty(\T)$ that properly contains the algebra, $\H$, of (boundary values) of bounded analytic functions 
in the open unit disk $\D=\{z\in\C: |z|<1\}$. We refer the reader to the book of J. B. Garnett 
\cite{ga} for items and results not explicitely defined here.  

The simplest example of a Douglas algebra is the algebra $\H+C$ of sums of (boundary values) of bounded analytic functions and complex-valued continuous functions on $\T$. This algebra is frequently called the Sarason algebra since it was first shown by Sarason, that this space is a closed subalgebra of $L^\infty(\T)$, see \cite{sa}. He showed that (on $\T$), $\H+C$ is the uniform closure of the set of functions  $\{f\ov z^n: f\in\H, n\in\N\}$.   Thus, $\H+C=[\H,\ov z]$, the closed algebra generated by
$\H$ and the monomials $\ov z^n$.

Let $M(\H)$ denote the maximal ideal space of $\H$.  It is well known that the spectrum of $M(\H+C)$ can be identified with the Corona of $\H$, namely the set $M(\H)\setminus\D$, see \cite[p. 377]{ga}.
 We denote by 
 $$
 Z(f)=\{m\in  M(\H): f(m)=0\},
 $$
  the zero set of a function in $\H$. 
 We will also need the notion of pseudo-hyperbolic distance, $\rho(x,m)$,
  of two points $m,x\in M(\H)$. Recall that
  \begin{linenomath}
  $$\rho(x,m)=\inf\{|f(m)|: ||f||_\infty\leq 1, f(x)=0\},$$
  \end{linenomath}
  and that for $z,w\in\D$ we obtain 
  $\rho(z,w)=\left| \frac{z-w}{1-\ov w z}\right|$, where we identify the point evaluation functional
  $f\mapsto f(z)$ with the point $z$ itself.
  
  Also, for a  set $E$ in $M(\H)$, let $\rho(E,x)=\inf\{\rho(e,x): e\in E\}$
  be the pseudo-hyperbolic distance of $E$ to a point $x\in M(\H)$.
  Similarly, $\rho(E,U)=\inf\{\rho(E,u): u\in U\}=\inf\{\rho(e,u):e\in E,u\in U\}$.
  It is well known that for closed sets $E$ the distance functions  $\rho(\cdot,\cdot)$
   and $\rho(\cdot, E)$  are lower semi-continuous
  on $M(\H)$, see \cite{gm} and \cite{hof1}. In particular, if $\rho(E,x)>\eta>0$, then
  there  exists an open set $U$ containing $x$ such that $\rho(E, U)>\eta$.

   Finally,  for a point $m\in M(\H)$, let $P(m)=\{x\in M(\H): \rho(x,m)<1\}$
  denote the Gleason part associated with $m$. 
  For example $\D$ itself is the Gleason part asociated with the origin.
  By Hoffman's theory,
   there exists a map $L_m$ of $\D$ onto $P(m)$ such that $\hat f\circ L_m$ is analytic
   for all $f\in \H$. If $(a_\beta)$ is any net in $\D$ that converges to $m$, then $L_m$
    is given by $L_m(z)=\lim \frac{a_\beta+z}{1+\ov {a_\beta} z}$,
   where the limit is taken in the topological product space $M(\H)^{\D}$.
  
  \bigskip

\subsection{The Corona Property for $\H+C$}

Our first theorem is based on Axler's result  that any function $u\in L^\infty$
can be multiplied by a \BP\ into $\H+C$.  See \cite{ax}.  Using this result it is possible to prove the following assertion.

     \begin{theorem}\label{corona}
   The Corona Theorem with bounds holds on  $\H+C$.   
      \end{theorem}
      
   \begin{proof}

    {\sl Step 1:} We first consider $\H$-Corona data on $M(\H+C)$.  The goal is to find $\H+C$ solutions.
   
   Let $f_1,\ldots, f_n \in \H$ satisfy $||f_j||_\infty\leq 1$
    and $ |f_1|+\cdots+|f_n|\geq\delta>0$ on $M(\H+C)$.
   In particular, $|f_1|+\cdots+|f_n|\geq \delta>0$ a.e. on $\T$. Hence,
   \begin{linenomath}
   $$1= \sum_{j=1}^nf_j\; \frac{\ov {f_j}}{\sum_{k=1}^n |f_k|^2}, \textnormal{ a.e. on } \T.$$
   \end{linenomath}
   Note that the functions $\frac{\ov {f_j}}{\sum_k |f_k|^2}$
   belong to $L^\infty(\T)$.
   By the Axler Multiplier Theorem, \cite{ax}, there exists a \BP\ $B$ such that for all $j\in\{1,\dots,n\}$
   $$\Phi_j:=B\frac{\ov {f_j}}{\sum_k |f_k|^2}\in \H+C.$$
     So, a.e. on $\T$, we have 
   $B=\sum_j \Phi_j f_j$. Moreover,  
  $||\Phi_j||_\infty \leq n/\delta^2$.

  By continuity, there exists an annulus $A_r:=\{r\leq |z|<1\}$ such that $|f_1|+\cdots+|f_n|\geq \delta/2$
   on $A_r$.  Choose a tail $B_1$ of $B$ such that $|B_1|\geq \delta/2$ on $\{|z|\leq r\}$.
   Let $b_1=B/B_1$. Note that $\ov b_1$, viewed as a function on $\T$, belongs to $\H+C$.
   
      Now consider the ideal $I(f_1,\cdots, f_n,B_1)$ in $\H$.  We obviously have that
      $|f_1|+\cdots+|f_n|+|B_1|\geq \delta/2$ on $\D$. Hence, by the $\H$-Corona Theorem, there exists
      a constant depending on $\delta$, $C(\delta)$, and functions  $x_1,\ldots, x_n ,t\in \H$ with 
      \begin{linenomath}
      $$||x_1||_\infty +\cdots+||x_n||_\infty +||t||_\infty \leq  C(\delta)$$
      \end{linenomath}
      such that $1=\sum_{j=1}^n x_jf_j+tB_1$  on $\D$, and thus almost everywhere on $\T$ we have $1=\sum_{j=1}^n x_jf_j+tB_1$. 
      Switching again to $\H+C$, we see that a.e. on $\T$
 
      \begin{linenomath}
    $$B_1=\ov b_1 B= \sum_{j=1}^n(\ov b_1\Phi_j)f_j.$$
      \end{linenomath}
      Hence,
      \begin{linenomath}
      $$1=\sum_{j=1}^n x_jf_j+t \biggl(\sum_{j=1}^n(\ov b_1\Phi_j)f_j \biggr)=\sum_{j=1}^n f_j\biggl(x_j+t\ov b_1\Phi_j\biggr),$$
     \end{linenomath}
     where $\sum_{j=1}^n||x_j+t \ov b_1\Phi_j||_\infty\leq  
     C(\delta)(1+n^2/ \delta^2)=:\chi(\delta)
     $.
       Since $\H+C$ is an algebra, the functions $\phi_j:=x_j+t \ov b_1\Phi_j$ belong to $\H+C$.
     Thus, we have found a solution with bounds to the Bezout equation
     $\sum_{j=1}^n \phi_j f_j =1$ a.e. on $\T$, or equivalently, on $M(\H+C)$.
         \medskip
     
     \medskip
     
     {\sl Step 2:} We now look at general $\H+C$ Corona data.
     
     Let $F_1,\ldots, F_n\in\H+C$ satisfy $||F_j||_\infty \leq 1$ and $\sum_{j=1}^n|F_j|\geq \delta>0$ on $M(\H+C)$.
     Let $\chi(\delta)$ be the function above. 
     We uniformly approximate $F_j$ by functions of the
      form $\ov z^{N_j} f_j$;
     say 
     \begin{linenomath}
     $$\sum_{j=1}^n||F_j-\ov z^{N_j} f_j||_\infty \leq \e=\e(\delta):=\min\{\delta/4, [4\chi(\delta/2)]^{-1}\},$$
     \end{linenomath}
     where $f_j\in \H$ and $||f_j||_\infty\leq 1$.
     
     Then $ \sum_{j=1}^n|f_j|\geq \delta/2$ on $M(\H+C)$. We apply Step 1 to the functions $f_j$ and find $\H+C$ functions $\phi_j$ and a constant $\chi(\delta/2)$ that bounds the sum of the  norms of these functions and such that  $1=\sum_{j=1}^n\phi_j f_j$. Now 
     \begin{linenomath}
     $$u:=\sum_{j=1}^nz^{N_j}\phi_j F_j  =\sum_{j=1}^nz^{N_j}\phi_j(F_j-\ov z^{N_j}f_j)+\sum_{j=1}^n\phi_j f_j $$
     \end{linenomath}
     \begin{linenomath}
     $$=\gamma +1,$$
     \end{linenomath}
     where $\gamma:=\sum_{j=1}^nz^{N_j}\phi_j(F_j-\ov z^{N_j}f_j)\in\H+C$ is majorized by $\e\sum_{j=1}^n||\phi_j||_\infty \leq 
     \e\,\chi(\delta/2) \leq 1/2$.
     Thus $|u|\geq 1/2$ on $M(\H+C)$; hence $u$ is invertible in $\H+C$ with $||u^{-1}||_\infty\leq 2$.
     We conclude that 
     \begin{linenomath}
     $$1=\sum_{j=1}^n(u^{-1}z^{N_j}\phi_j)F_j$$
     \end{linenomath}
     where the coefficients are bounded by $2\chi(\delta/2)$.    
   \end{proof}
   
   An immediate Corollary (at least for the upper bound)  is the following.
   
   \begin{corollary}\label{niko}
   For every integer $n$ we have $\delta_n(\H+C)=0$ and, for $\delta$ close to $0$,
  \begin{linenomath}
$$\kappa\;\delta^{-2}\log(\log (1/\delta))\leq c_n(\delta,\H+C)\leq 2\chi\left(\frac{\delta}{2}\right),$$
   \end{linenomath}
   where $\chi(\delta)=(1+\frac{n^2}{\delta^2})C(\delta)$ with $C(\delta)$ the best constant in the $\H$-Corona Problem and where $\kappa$ is an absolute constant.
   \end{corollary}
     
   \begin{proof}
  It remains to verify the lower estimate.  Here we use a result of Treil \cite[p. 484]{tr2},
   that tells us that there exist two finite \BP s $B_1$ and $B_2$ satisfying
   $|B_1|+|B_2|\geq \delta>0$ on $\D$ such that for any solution $(g_1,g_2)\in (\H)^2$  of the Bezout
   equation $g_1B_1+g_2B_2=1$ we have $||g_1||_\infty\geq \kappa\, \delta^{-2}\log(\log (1/\delta))$
   whenever $\delta>0$ is close to $0$.
   Now, of course, this does not give us an example in $\H+C$, since $1=\ov B_1 B_1+ 0 B_2$
   is a solution with coefficients bounded by 1. We proceed to the following modification.
  
   Let $m$ be a thin point in  $M(\H+C)$, that is a point lying in the
   closure of a thin interpolating sequence, say $(z_n)=(1-1/n!)$. Since the associated
   \BP\ $b$ satisfies $(1-|z_n|^2)|b'(z_n)| \to 1$, Schwarz's Lemma implies  that
   $(b\circ L_m)(z)=e^{i\theta}z$ for every $z\in\D$. Now consider the functions
   $$\mbox{$f_1=B_1\circ(e^{-i\theta} b)$ and 
   $f_2=B_2\circ (e^{-i\theta}b)$}.$$
   Clearly $|f_1|+|f_2|\geq \delta$ on $\D$ and hence, viewed as functions
   in $\H+C$, we have $|f_1|+|f_2|\geq \delta$ on $M(\H+C)$.  Let $(h_1,h_2)\in (\H+C)^2$
   be  a solution  of $h_1f_1+h_2f_2=1$ in $\H+C$.
   Since  $f_j\circ L_m=B_j$, we get that $1=(h_1\circ L_m)B_1+(h_2\circ L_m)B_2$
   in $\D$. Thus, by Treil's result mentioned above, 
   $$||h_1||_\infty\geq ||h_1\circ L_m||_\infty\geq \kappa\, \delta^{-2}\log(\log (1/\delta)).$$
  Thus, 
   $$c_n(\delta, \H+C)\geq c_2(\delta, \H+C)\geq \kappa\, \delta^{-2}\log(\log (1/\delta)).$$
   
      \end{proof}

   \section{Several Notions of Stable Ranks}\label{two}
   
   Let $A$ be  a commutative unital ring.  An $n$-tuple $(a_1,\dots, a_n)\in A^n$
   is said to be {\sl invertible} (or unimodular)  if there is a solution  $(x_1,\dots, x_n)\in A^n$
   of the Bezout equation  $\sum_{j=1}^n  a_jx_j=1$.  Of course this is equivalent to saying
   that the ideal generated by the $a_j$ is the whole ring. The set of all invertible $n$-tuples in $A$
   is denoted by $U_n(A)$.
   
   An $(n+1)$-tuple $(a_1,\dots, a_n, a_{n+1})\in U_{n+1}(A)$ is said to be {\sl $n$-reducible} 
   (or simply reducible) in $A$ if
   there exists $(x_1,\dots,x_n)\in A^n$ such that $(a_1+x_1 a_{n+1},\dots, a_n+x_na_{n+1})$
   is an invertible
    $n$-tuple in $A^n$. It can be shown that if every invertible $n$-tuple in $A$ is reducible,  then
    every invertible $(n+1)$-tuple is reducible (see for example \cite{va}).
     The smallest  integer $n\in \{1,2,\dots\}$,  for which every invertible $(n+1)$-tuple is reducible is called the {\it Bass stable rank} of $A$ and is denoted by 
    $\text{bsr}(A)$.  This notion was introduced in $K$-theory.

    For algebras of continuous or analytic functions, this has been studied e.g. by 
     Corach and Larotonda \cite{cl}, Rupp \cite{ru1, ru2, ru3, ru4}, Su\'arez  \cite{su0, su}
      and Vasershtein \cite{va}.  It was shown by Jones, Marshall and Wolff, see 
    \cite{jmw}, that the Bass stable rank of the disk algebra $A(\D)$ is one.    Later, simpler proofs
    were given by Corach, Su\'arez  \cite{cs1} and Rupp \cite{ru1, ru2}.     Treil later showed that  
     the Bass stable rank of $\H$ is one, \cite{tr}.

   A notion related to the Bass stable rank is that of the topological stable rank. 
    Let $A$ be a commutative unital Banach algebra.  The smallest integer $n$ for which
    the set, $U_n(A)$, of invertible $n$-tuples  is dense in $A^n$ is called the
    {\it topological stable rank} of $A$, denoted by $\text{tsr}(A)$.  This notion was introduced by Rieffel, \cite{ri}, in the study of $C^\ast$-algebras. 
      
Finally, we recall two additional  notions of stable ranks.
Let $\B$ be the class of all commutative unital Banach algebras over a field $\K$.  We will always assume that   algebra homomorphisms, $f$, between members of $\B$
are continuous
and  satisfy $f(1_A)=1_B$. Also, if $f:A\to B$ is an algebra homomorphism, then $\underline{f}$ will denote the associated map given by $\underline f: (a_1,\dots,a_n)\mapsto (f(a_1),\dots,f(a_n))$ from $A^n$ to $B^n$.  

By \cite[p. 542]{cs3}, the {\it dense stable rank} $\text{dsr}(A)$ of  $A\in \B$ 
is the smallest integer $n$ such that for every $B\in\B$ and every  algebra homomorphism
$f:A\to B$  with dense image the induced map $U_n(A)\to U_n(B)$ has dense image.
If there is no such $n$, we write $\text{dsr}(A)=\infty$.
We note that if for some $n\in\N$, all $B\in\B$ and all homomorphisms $f: A\to B$ with dense image
the set $\underline{f}(U_n(A))$ is dense in $U_n(B)$, then $\underline{f}(U_{n+1}(A))$
is dense in $U_{n+1}(B)$, (see   \cite[p. 543]{cs3}).

The {\it surjective stable rank} $\text{ssr}(A)$ of  $A\in \B$  is 
the smallest integer $n$ such that for every $B\in\B$ and every surjective algebra homomorphism
$f:A\to B$ the induced map of  $U_n(A)\to U_n(B)$ is surjective, too.
Again, if there is no such $n$, then we write $\text{ssr}(A)=\infty$. 
Let us point out, that the assumption $\underline{f}(U_n(A))=U_n(B)$ for some $n$, all $B\in\B$ and all surjective homomorphisms $f:A\to B$
 implies that
$\underline{f}(U_{n+1}(A))=U_{n+1}(B)$. This works similarly as  the proof for the corresponding statement
for denseness in \cite[p. 543]{cs3}.

In fact, let $b:=(b_1,\dots,b_{n+1})\in U_{n+1}(B)$.
Consider  for $I=\ov{(b_{n+1})}$ (the closure of the principal ideal generated by $b_{n+1}$)
the quotient algebra $\tilde B:= B/I$ and the quotient mapping
$\pi: B\to \tilde B$.
Then $$(b_1+I,\dots,b_{n}+I)\in U_n(\tilde B).$$
 By our hypothesis, since $\pi f$ is surjective, there exists 
$a:=(a_1,\dots, a_n) \in U_n(A)$ such that  $\pi f(a_j)=\pi b_j$ for $1\leq j\leq n$.
Choose $\e>0$ so that every perturbation $(a_1-r_1, \dots, a_n-r_n)$ of $a$
with $||r_j||_A< \e$ is in $U_n(A)$ again. Using the open mapping theorem, 
let $\eta=\eta(\e)$ be so that
$$\{y\in B: ||y||_B<\eta\}\ss f(\{x\in A:||x||_A<\e\}).$$
Since $f(a_j)-b_j\in I$, there exist
$k_j\in B$ such that  $||f(a_j)-b_j-k_j b_{n+1}||_B<\eta$ ($j=1,\dots,n$). 
Since $f$
is surjective, we may choose $r_j, x_j\in A$, $||r_j||_A<\e$,  and $a_{n+1}\in A$ such that
$f(r_j)=y_j:=f(a_j)-b_j-k_j b_{n+1}$, 
$f(x_j)=k_j$,  ($j=1,\dots,n$),  and $f(a_{n+1})=b_{n+1}$. Then  
$$(a_1-r_1,\dots, a_n-r_n)\in U_n(A)$$ and so
$$(a_1',\dots,a'_{n+1}):=(a_1-r_1-x_1a_{n+1},\dots, a_n-r_n-x_na_{n+1}, a_{n+1})\in U_{n+1}(A).$$
Moreover,  $f(a_j')= f(a_j)-f(x_j)f(a_{n+1})-f(r_j)=f(a_j)-k_jb_{n+1}-y_j=b_j$ for $1\leq j\leq n$
and $f(a'_{n+1})=b_{n+1}$.  Hence $\underline{f}(U_{n+1}(A))=U_{n+1}(B)$.

\medskip

Next, we present relations between these notions of stable rank.  Most of these relations are known and can be found in the papers of Corach and Larotonda, see \cite{cl}, and Su\'arez, see \cite{su0, su}.  Many of the proofs in these papers are based on far reaching concepts and techniques from algebraic geometry, such as Serre fibrations and homotopy classes.  For the reader's convenience we present short direct proofs of some of these facts.

\begin{proposition}\label{st-top}\cite[p. 293]{cl}
Suppose that $U_n(A)$ is dense in $A^n$. 
Then the stable rank of $A$ is less than $n$. Namely, $\text{bsr}(A)\leq \text{tsr}(A)$.
\end{proposition}

\begin{proof}

Let $(f_1,\ldots,f_n,h)\in U_{n+1}(A)$.  Then there exist $x_j\in A$ and $x\in A$
so that $1=\sum_{j=1}^n x_j f_j +  xh$. Since $U_n(A)$ is dense in $A^n$, for every $\e>0$,
there exists $(u_1,\ldots,u_n)\in U_n(A)$ so that $||u_j-x_j||_A<\e$. Also
$x=\sum_{j=1}^n h_j u_j$ for some $h_j\in A$ because $(u_1,\ldots, u_n)$ is invertible.  Hence
$$
\sum_{j=1}^n u_j(f_j+h_jh)=\sum_{j=1}^n u_jf_j+xh
$$
$$
=\bigl(\sum_{j=1}^n x_j f_j +xh\bigr) +\sum_{j=1}^n (u_j-x_j)f_j=1 +u,
$$
where we have defined $u:=\sum_{j=1}^n (u_j-x_j)f_j$. Moreover, we have $||u||_A\leq \e\sum_{j=1}^n ||f_j||_A$.
Hence for $\e>0$ small enough,  $1+u$ is invertible in $A$, and so $(f_1+h_1h,\dots,f_n+h_nh)\in 
U_n(A)$.
\end{proof}

The following lemma is due to Corach and Su\'arez \cite{cs1,cs2}.
  \begin{lemma}\label{clopen-n}\cite[p. 636]{cs1} and \cite[p. 608]{cs2}
   Let $A$ be  a commutative unital Banach algebra. Then, for $g\in A$, the set 
$$
\mbox{$R_n(g)=\{(f_1,\dots, f_n)\in A^n: (f_1,\dots,f_n, g)$ is reducible $\}$}
$$
 is open-closed inside 
 $$I_n(g)=\{(f_1,\dots, f_n)\in A^n: (f_1,\dots, f_n, g)\in U_{n+1}(A)\}.$$
 In particular, for $n=1$,
if $\phi: [0,1]\to I_1(g)$ is a continuous curve and $(\phi(0),g)$ is reducible, then $(\phi(1),g)$
is reducible.
\end{lemma}

A very useful characterization of the Bass stable rank is the following. 
Here the equivalence of items (2) and (3) was  known (see \cite{cl}).


\begin{theorem}\label{bsrdsr}
Let $A$ be  a commutative unital Banach algebra. The following assertions are equivalent:
\begin{enumerate}
\item $\underline{\pi}(U_n(A))$ is dense in $U_n(A/I)$ for every closed ideal $I$ in $A$;
\item $\text{bsr} (A)\leq n$;
\item $\underline{\pi}(U_n(A))=U_n(A/I)$ for every closed ideal $I$ in $A$.
\end{enumerate}
 Here $\pi:A\to A/I$ is the canonical quotient mapping and $\underline{\pi}$ the associated map
 on $A^n$.
\end{theorem}

\begin{proof}

(1) $\imp$ (2):  Let $(a_1,\dots,a_n,a_{n+1})\in U_{n+1}(A)$. Consider the closure, $I$, of the
 ideal generated by $a_{n+1}$. Then $A/I$ is a Banach algebra under the quotient norm and
 $(a_1+I,\dots, a_n+I)\in U_n(A/I)$. By (1), 
 there exists a sequence  $(b^{(k)}_1,\dots,b^{(k)}_n)\in U_n(A)$ such that 
 $||\pi(b^{(k)}_j)- \pi(a_j)||_{A/I}\to 0$ as $k\to\infty$ for
 $j=1,\dots, n$.  Hence there are $x^{(k)}_j\in A$ so that 
 $$||a_j-b^{(k)}_j+x^{(k)}_ja_{n+1}||_A\to 0.$$
 Now for every $k$ we have that  the $n+1$-tuples
$$(b^{(k)}_1-x^{(k)}_1a_{n+1},\dots,b^{(k)}_n-x^{(k)}_na_{n+1}, a_{n+1})$$
 are invertible and reducible, since 
 $(b^{(k)}_j-x^{(k)}_ja_{n+1})+x^{(k)}_ja_{n+1}=b^{(k)}_j$
 and $(b^{(k)}_1,\dots,b^{(k)}_n)\in U_n(A)$.
 Using  Lemma \ref{clopen-n}, which tells us that $R_n(a_{n+1})$ is closed inside 
 $I_{n}(a_{n+1})$, and noticing that $b^{(k)}_j-x^{(k)}_ja_{n+1}\to a_j$ for $j=1,\dots,n$,
 we see that  $(a_1,\dots, a_n, a_{n+1})$ is reducible and so $\text{bsr} (A)\leq n$.

 \medskip

(2) $\imp$ (3): This appears in \cite[p. 296]{cl}. For the reader's convenience we present the argument.  Let $(a_1+I,\dots,a_n+I)\in U_n(A/I)$. Then there exist $y_1,\dots, y_n\in A$ and $b\in I$
such that $\sum_{j=1}^n y_ja_j =1+b$. Hence $(a_1,\dots, a_n,b)\in U_{n+1}(A)$ and
so, by (2), there exists $x_1,\dots,x_n\in A$ such that
$$(a_1+x_1b,\dots, a_n+x_nb)\in U_n(A).$$
It is clear that $\pi(a_j+x_jb)=a_j+I$. Hence $\underline{\pi}(U_n(A))=U_n(A/I)$.

\medskip

(3)$\imp$ (1): This is immediate. 
\end{proof}

Parts of the following result  appear without proof in \cite[p. 542]{cs3}.

\begin{theorem}\label{ranks}
If $A$ is a commutative unital Banach algebra, then 
$${\rm bsr}(A)={\rm ssr}(A)\leq {\rm dsr}(A)\leq {\rm tsr}(A).$$
\end{theorem}
\begin{proof}

The assertion  that $\text{bsr}(A)=\text{ssr}(A)$ follows from Theorem \ref{bsrdsr}.
Indeed, let $n=\text{ssr}(A)<\infty$. Since for any closed ideal $I$ the canonical map $\pi: A\to A/I$
is surjective,  $\text{ssr}(A)=n$ implies that
$\underline{\pi}(U_n(A))=U_n(A/I)$. Hence, by Theorem \ref{bsrdsr}, $m:=\text{bsr}(A)\leq n$.
To show that $n\leq m$, 
let  $f:A\to B$ be a surjective homomorphism. Then
the canonical injection  $\check f: \tilde A=A/ \text {Ker} f\mapsto B$ is an algebra isomorphism
and so $U_m( \tilde A)$ is mapped onto $U_m(B)$ by $\underline {\check f}$.
Since $ m=\text{bsr}(A)$,   by Theorem \ref{bsrdsr}, 
$\underline{\pi}(U_m(A))=U_m(A/\text {Ker} f)$.
Thus $\underline f(U_m(A))=U_m(B)$. This means that $\text{ssr}(A)\leq m$.
All together we have shown that $\text{bsr}(A)=\text{ssr}(A)$.

Now suppose that $\text{dsr}(A)=n$. Consider the algebra $B:=A/I$, where $I$
is any closed ideal in $A$. Then the assertion that $\text{bsr}(A)\leq \text{dsr}(A)$
follows from Theorem \ref{bsrdsr} when applied to the epimorphism $f=\pi$.

Next we suppose that $\text{tsr}(A)= n$.  To show that $\text{dsr}(A)\leq \text{tsr}(A)$, we note that if $f:A\to B$ has dense image, then
$\underline f: A^n\to B^n$  has dense image as well. Now, the continuity of $f$
and the density of $U_n(A)$ in $A^n$ imply that 
$$ \ov{\underline f(U_n(A))} \supseteq \underline f\hspace{-3pt}\left(\ov{U_n(A)}\right)=
 \underline f(A^n).$$
Therefore, $$U_n(B)\ss B^n=\ov{\underline f(A^n)}\ss \ov {\underline f(U_n(A))}.$$
Since $\underline f(U_n(A))\ss U_n(B)$, we finally obtain that
 $\underline f(U_n(A))$ is dense in $U_n(B)$. Therefore $\text{dsr}(A)\leq n=\text{tsr}(A)$.
 
\end{proof}

It is not known whether always  $\text{bsr}(A)=\text{dsr}(A)$. For 
 $A=\H$, for instance, we have $\text{bsr}(\H)=\text{dsr}(\H)=1$ (see \cite{nisu, su0, tr})
and $\text{tsr}(\H)=2$ (see \cite{su});
for $A=\H_\R=\{f\in\H: \ov{f(\ov z)}=f(z)\}$ we have
$\text{bsr}(\H_\R)=\text{dsr}(\H_\R)=\text{tsr}(\H_\R)=2$,  (see \cite{mw}).

\bigskip
 
Our next result, which seems to be new, gives a version of Lemma \ref{clopen-n} with bounds.
\begin{proposition}\label{bounds}
Let $(f,g)$ be an invertible pair in the commutative unital Banach algebra 
$A$. Suppose that  $f_n$ converges to $f$ and 
 that there exist a constant $K\geq 1$ and $u_n\in A$ such that  $f_n+u_ng$
is invertible with 
\begin{linenomath}
$$||u_n||_A+||f_n+u_ng||_A+||(f_n+u_ng)^{-1}||_A\leq K.$$
\end{linenomath}
Then  there is $h\in A$ such that $f+hg$ is invertible and
\begin{linenomath}
$$ ||h||_A+||f+hg||_A+||(f+hg)^{-1}||_A\leq 8K.$$
\end{linenomath}
\end{proposition}

\begin{proof}
Let $(x,y)\in A^2$ be so that $1=xf+yg$. Then
\begin{linenomath}
$$f_n+u_ng=f(xf_n+yg)+\bigl( y(f_n-f)+u_n\bigr) g.$$
\end{linenomath}
Since $||f_n-f||_A\to 0$ we may choose $n_0$ so big that for all $n\geq n_0$ the elements $xf_n+yg$
are invertible in $A$,  that
\begin{linenomath}
$$\mbox{$ ||xf_n+yg||_A\leq 2$ and $||(xf_n+yg)^{-1}||_A\leq 2$},$$
\end{linenomath}
and so that 
\begin{linenomath}
$$||f-f_n||_A\leq \min\{1,||y||_A^{-1}\}.$$
\end{linenomath}
If we let 
\begin{linenomath}
$$h= \bigl( y(f_n-f)+u_n\bigr) (xf_n+yg)^{-1},$$
\end{linenomath}
 then $f+hg$ is invertible.
 Noticing that by hypothesis $||u_n||_A\leq K$, we get
 \begin{linenomath}
 $$||h||_A\leq 2 (1+K), $$
 \end{linenomath}
  \begin{linenomath}
 $$||(f_n+u_ng)(xf_n+yg)^{-1}||_A\leq 2 K$$
 \end{linenomath}
  and
  \begin{linenomath}
  $$||(xf_n+yg)(f_n+u_ng)^{-1}||_A\leq 2 K.$$
  \end{linenomath}
  Since $K\geq1$,
  \begin{linenomath}
  $$||h||_A+||f+hg||_A+ ||(f+hg)^{-1}||_A\leq 8K.$$
  \end{linenomath}
 \end{proof}

\section {The Stable Ranks of  $\H+C$} \label{three}

It is the aim of this section to determine the stable ranks defined above for
the Sarason algebra $\H+C$.  To this end let us recall the 
following Theorem of S. Treil \cite{tr} that tells us in particular that
$\text{bsr}(\H)=1$.

\begin{theorem}[Treil]\label{treil}
 There exists a constant $C(\delta)$ depending only on $\delta\in ]0,1[$
 such that for every pair $(f,g)$ of elements in the unit ball of
 $\H$ satisfying
 $|f|+|g|\geq \delta>0$ in $\D$,  there are functions $u,h\in\H$ with 
  $u$ invertible in $\H$, such that
 $1=uf+hg$ and $||u||_\infty+||u^{-1}||_\infty+||h||_\infty\leq C(\delta)$.
\end{theorem}

   The following well known result can easily be deduced from Treil's Theorem.
   \begin{proposition}
   There exists a constant $C(\delta)$  depending only on $\delta$ such that for every
    $f_1,\dots,f_n$ in $\H$ satisfying $1\geq\sum_{j=1}^n |f_j|\geq \delta>0$  in $\D$, there exist
     $a_j, t_j\in \H$
    bounded by $C(\delta)$ so that 
    \begin{linenomath}
    $$1=\sum_{j=1}^{n-1}a_j(f_j+ t_jf_{n}).$$
    \end{linenomath}
   \end{proposition}
   
   \begin{proof}
   By the $\H$-Corona Theorem, \cite{c}, there is a constant $C_1(\delta)$ such that the Bezout equation
   $\sum_{j=1}^n x_jf_j=1$ admits a solution $(x_1,\ldots, x_{n})\in({\H})^{n}$ with 
   \begin{linenomath}
    $$\sum_{j=1}^{n}||x_j||_\infty \leq C_1(\delta).$$
    \end{linenomath}
      We may assume that $C_1(\delta)\geq 1$.
    Now 
     \begin{linenomath}
    $$1\leq ||x_1||_\infty  |f_1| +\left\vert\sum_{j=2}^{n}x_j f_j\right\vert\leq C_1(\delta)\left( |f_1|+\left\vert\sum_{j=2}^n x_jf_j\right\vert\right) ;$$ 
    \end{linenomath}
    hence 
    $$2+C_1(\delta)\geq |f_1|+\left\vert\sum_{j=2}^nx_j f_j\right\vert\geq \frac{1}{ C_1(\delta)}:=\e>0.$$
    By Treil's Theorem there  is a constant  $C_2(\e)$ and
    $u,v\in\H$ such that $u$ is invertible in $\H$, $||u||_\infty +||u^{-1}||_\infty +||v||_\infty \leq C_2(\e)$
   with 
   $$
   1=u f_1 +v(\sum_{j=2}^nx_j f_j).
   $$
   The latter equation can be re-written as
  \begin{linenomath}
   $$ 1=u(f_1+u^{-1}vx_n f_n)+ \sum_{j=2}^{n-1} vx_j(f_j+0\cdot f_n).$$
   \end{linenomath}
   
   It is clear that the functions $a_1:=u$, $t_1:=u^{-1}vx_n$, $a_j:= vx_j$ and $t_j:=0$ for $j=2, \ldots, n-1$ are bounded by a constant $C(\delta)$ depending only on $\delta$.

   \end{proof}

   The following Theorem given by Laroco \cite[p. 819]{la} will be essential for our
   determination of the stable rank of $\H+C$.
   
   \begin {theorem} [Laroco]\label{laro}
   Let $f\in\H$. Then, for every $\e>0$,  there exist a \BP\ $B$ and an outer function, $v$,
   invertible in $\H$ such that
   \begin{linenomath}
   $$\mbox{$||f-Bv||_\infty<\e$ and $ |v|\geq \e/4$ $\textnormal{ on } \partial\mathbb{D}$,}$$
  as well as $||v||_\infty\leq 1+||f||_\infty$.
   \end{linenomath}
   \end{theorem}
   
   We note that, due to the fact that $v$ is invertible, we actually have  $|v|\geq \e/4$
   on $\D$.
    We additionally need the following lemma.
    
   \begin{lemma}\label{special}
  Let $B$ be a Blaschke product and $g\in\H$ with norm less than or equal to one.  Suppose that  $|B|+|g|\geq \delta>0$ on $M(\H+C)$.  Then there exists a constant $C(\delta)$, depending only on $\delta$, and functions $h$ and $u$ in $\H+C$ with $u$ invertible in $\H+C$ so that
   $||u||_\infty , ||u^{-1}||_\infty $ and $||h||_\infty $ are bounded by $C(\delta)$ and such that
   \begin{linenomath}
   $$ 1=uB+h g.$$
   \end{linenomath}
   \end{lemma}
   \begin{proof}
    By continuity, there exists  $r>0$ such that on $\{r\leq |z| <1\}$
    \begin{linenomath}
$$|B|+|g|\geq \delta/2.$$
\end{linenomath}

    Let $B^*$ be  a tail of $B$, so that $|B^*|\geq \delta/2$ on $\{|z|\leq r\}$.
    Hence
    \begin{linenomath}
    $$\mbox{$|B^*|+|g|\geq \delta/2$ on $M(\H)$}.$$
    \end{linenomath}
    Let $b=B/B^*$. Note that $b$ is a finite \BP\ and hence $\ov b\in\H+C$.
    By Treil's result \cite{tr} (here Theorem \ref{treil}),
     there is a constant $C(\delta)$ and two functions 
    $R,h\in\H$, $R$ invertible in $\H$,   such that 
     \begin{linenomath}
    $$1=RB^*+hg$$
    \end{linenomath}
    and 
    \begin{linenomath}
    $$ ||R||_\infty +||R^{-1}||_\infty + ||h||_\infty \leq C(\delta/2).$$
    \end{linenomath}
    
   Therefore, as  $\H+C$-functions: 
   \begin{linenomath}
   $$1= (R\ov b) B+ h g.$$
   \end{linenomath}
    \end{proof}

   \begin{theorem}\label{bsr=1}
   The Bass stable rank of $\H+C$ equals one.
   \end{theorem}
   
   \begin{proof}
   Let $(\phi,\psi)$ be an invertible pair in $\H+C$. We may assume that  
   $||\phi||_\infty\leq 1$ and $||\psi||_\infty\leq 1$.
   Since $\H+C$ is the uniform closure of the set of functions
    $\{\ov z^n f: f\in\H, n\in \N\}$
        (see \cite{ga}), there exists $n\in\N$ and $f\in \H$  such that $||\phi-f \ov z^n||_\infty <\e$.
        By Theorem \ref{laro}, there is a \BP\ $B$ and a function $v$ invertible in
        $\H$ such that $||f-vB||_\infty<\e$. Hence the set of functions 
        \begin{linenomath}
        $$\mbox{$\{\ov z^n  vB: n\in\N, B$ Blaschke, $v$ invertible in $\H\}$}$$
        \end{linenomath}
        is dense in $\H+C$. By Lemma \ref{clopen-n} and the fact that the factors 
        $\ov z^n v$ are invertible in $\H+C$,   it suffices to show the 
        reducibility of the pairs $(B, \psi)$, where $B$ is any \BP\ such that $|B|+|\psi|\geq\delta>0$
        on $M(\H+C)$.
        
        To do this, we shall use Lemma \ref{special}. Choose  $n\in\N$ and $g\in \H$, $||g||_\infty\leq 1$,
         so that 
        \begin{linenomath}
        $$||\ov z^n g-\psi||_\infty < \min\biggl\{\frac{\delta}{2}, \frac{1}{2C(\delta/2)}\biggr\},$$
        \end{linenomath}
         where $C(\delta)$ is the constant from 
        Lemma  \ref{special}. 
        Now consider the pair $(B, g)$. We obviously have  (on $M(\H+C)$)
        \begin{linenomath}
        $$|B|+|g|= |B|+|\ov z^n g|\geq |B|+|\psi| - |\psi-\ov z^n g|\geq \delta/2.$$
        \end{linenomath}
        By Lemma \ref{special}, there exists $u\in \H+C$, $u$ invertible, and $h\in\H+C$
        with $||u||_\infty+||u^{-1}||_\infty+||h||_\infty\leq C(\delta/2)$ such that
        \begin{linenomath}
        $$1=uB+hg.$$
        \end{linenomath}
        Hence 
        \begin{linenomath}
        $$ |uB +(h z^n)\psi|=|uB+h g+h(z^n\psi-g)|\geq$$
         \end{linenomath}
         \begin{linenomath}
        $$ 1- ||h||_\infty\, ||\psi-\ov z^n g||_\infty\geq
        1-||h||_\infty\,  \frac{1}{2 C(\delta/2)} \geq 1/2.$$
        \end{linenomath}
        
        Thus $uB+(hz^n)\psi$ is invertible in $\H+C$. Hence
         $1= xB+y\psi$ where $x\in\H+C$ is invertible and 
        \begin{linenomath}
        $$\mbox{$\max\{||x||_\infty,  ||y||_\infty\}\leq 2 C(\delta/2)$ and  $||x^{-1}||_\infty\leq 2C^2(\delta/2)$}.$$
      This shows that the pair $(B,\psi)$ is reducible in $\H+C$.
        \end{linenomath}
   \end{proof}
   
  Combining Proposition \ref{bounds} with the proof above, we get the following
   extension of Theorem \ref{bsr=1}.

   \begin{theorem}
   There exists a constant $C(\delta)$ depending only on $\delta\in\;]0,1[$ such
   that for every pair $(\phi,\psi)$ of functions in the unit ball of 
   $\H+C$ satisfying $|\phi|+|\psi|\geq \delta$
   on $M(\H+C)$, there is a solution $(u,v)\in (\H+C)^2$ of the Bezout equation
   $u\phi+v\psi=1$,  where $u$ is invertible in $\H+C$, and such that
    $||u||_\infty+||u^{-1}||_\infty+||v||_\infty\leq C(\delta)$.
     \end{theorem}
     
     \begin{proof}
     According to Theorem \ref{corona}, respectively Corollary \ref{niko}, 
     let $(x,y)$ be a solution in $\H+C$ of the Bezout equation $x\phi+y\psi=1$
     with $||x||_\infty+||y||_\infty\leq \tilde\chi(\delta)$. 
     Using Theorem \ref{laro}, we  may choose a \BP\ $B$ and a function $h$ invertible in $\H$
      such that 
      $$||\phi-\ov z^{\nu} Bh||_\infty<\sigma(\delta):=\min\{\delta/2, \bigl(2\tilde\chi(\delta)\bigr)^{-1}\}$$
      and $2\geq |h|>\sigma(\delta)/4$.     Since $|B|+|\psi|\geq \delta/4$ on $M(\H+C)$,
       there   exists by the proof
      of Theorem \ref{bsr=1} above a constant $C_1(\delta)$ and a function $q\in\H+C$
      such that $ F:=B+q\psi$ is   invertible in $\H+C$ and so that 
      $$||q||_\infty  +||F||_\infty+||F^{-1}||_\infty\leq C_1(\delta).$$ 
     
      Let $f^*:=\ov z^{\nu}hB$ and $u^*:=\ov z^{\nu}hq$. Then $v^*:= f^*+u^*\psi$
      is invertible in $\H+C$ and 
      $$||u^*||_\infty+||v^*||_\infty+||(v^*)^{-1}||_\infty\leq C_2(\delta):=6C_1(\delta)+\frac{4}{\sigma(\delta)} C_1(\delta).$$
      Now, as in the proof of Proposition \ref{bounds},
      we see from
      $$ f^*+u^*\psi =\phi(xf^*+y\psi) + \bigl( y(f^*-\phi)+u^*\bigr) \psi$$
      that $1=u\phi+v\psi$ with
      $$||u||_\infty+||u^{-1}||_\infty+||v||_\infty\leq C_3(\delta).$$
     
      \end{proof}
  
  In \cite{su}, Su\'arez showed that  $\text{tsr}(\H)=2$. Using this result we deduce
  the topological stable rank for $\H+C$.

   \begin{theorem}\label{tsr}
   The topological stable rank of $\H+C$ is 2.
   \end{theorem}
   
   \begin{proof}
   First we show that the topological stable rank of $\H+C$ is at most 2.
   Let $(\phi_1,\phi_2)\in (\H+C)^2$. Approximate $\phi_j$ by functions
    of the form $\ov{z^{n_j}}f_j$, where $f_j\in\H$,
   say $||\ov{z^{n_j}}f_j-\phi_j||_\infty<\e$,  $j=1,2$.
   
    Since the topological stable rank of $\H$ is two, there exists $g_j\in\H$
   such that $||g_j-f_j||_\infty\leq \e$ and so that $(g_1,g_2)\in U_2(\H)$.
   Obviously $(\ov{z^{n_1}}g_1, \ov{z^{n_2}}g_2)\in U_2(\H+C)$ and
   $||\ov{z^{n_j}}g_j-\phi_j||_\infty<2\e$. Thus ${\rm tsr}\; (\H+C)\leq 2$.
      
   Let $b$ be an (infinite)  \IBP. Note that $b$ is not invertible in $\H+C$. Let $m\in M(\H+C)$
   be a zero of $b$ and let $L_m$ be the associated Hoffman map.
   Since $b\circ L_m$ is analytic, but not identically zero,  we see  that  
    $b$ cannot be uniformly approximated on $M(\H+C)$ by invertibles in $\H+C$, 
    say $u_n$, since otherwise  
   $||b\circ L_m-u_n\circ L_m||_\infty\to 0$; a contradiction to Rouch\'e's theorem.
   Thus ${\rm tsr}\; (\H+C)\geq 2$.
    \end{proof}

    Next we deal with the dense stable rank of $\H+C$.
    Recall that if $A$ is a commutative unital Banach algebra, then $E$ is an $A$-convex
     subset of $M(A)$ if
     $$\forall x\notin E\; \exists f\in A: |f(x)|>\sup_E |f|.$$
     We let $\hat E$ denote the $A$-convex hull of a closed set $E\ss M(A)$.  This is given by
     $$\hat E=\{m\in M(A): |f(m)|\leq \sup_E|f|\;,\; \forall f\in A\}.$$
     Note that $E$ is $A$-convex if and only if $E=\hat E$.
     We say that $E$ is a proper $A$-convex set if $E=\hat E$ and $\hat E\not=M(A)$.


    \begin{theorem}\label{dsr=1}
    The dense stable rank of $\H+C$ is one.
     \end{theorem}
      \begin{proof}
     Let $E$ be an $\H+C$-convex subset of $M(\H+C)$.
     By \cite{nisu} it is sufficient to prove that any function $\phi\in\H+C$ that does not 
     vanish on $E$ can be uniformly approximated on $E$ by an invertible function
     in $\H+C$.  To this end, we first approximate $\phi$ on $\T$ (hence on $M(\H+C)$)
      by  a function
     of the form  $\ov z^n f$, where $n\in \N$ and $f\in\H$,
     say $||\phi-\ov z^n f||_\infty<\e/2$.  By choosing $\e$ sufficiently small, we see 
     that $f$ does not vanish on $E$, too. 
     Let $$\check E=\{m\in M(\H): |h(m)|\leq \sup_E |h|, \;\forall h\in \H\}$$
    be the  $\H$-convex hull of $E$. 
  We show that $\check E\inter M(\H+C)=E$. To this end, 
  let $m_0\in M(\H+C)\setminus E$. 
Since $E$ is $\H+C$-convex, there
exists a function $\psi\in\H+C$ such that $|\psi(m_0)|>\sup_E |\psi|$. Uniformly approximating
$\psi$ by a function of the form $\ov z^n g$, 
shows that $|g(m_0)|>\sup_E |g|$ for some $g\in\H$.
Thus $m_0\not\in\check E$, the $\H$-convex hull of $E$. 
This shows that $\check E\inter M(\H+C)=E$.

Now $\check E$ can be written as 
$\check E=E\union S$, where $S=\check E\inter \D$. We probably have $S=\emp$, however this isn't necessary for the rest of the proof.
 Note that $\check E$ is closed and $\ov S\setminus \D\ss E$.
 Recall that our function $f$ above does not vanish on $E$; hence $f$ can have only 
 finitely many zeros in $S$.  Write $f=qF$, where $q$ is the finite \BP\
 formed with the zeros of $f$ in $S$. Thus $F$ does not vanish on $\hat E$.
By \cite{nisu}, for every $\e>0$, 
     there is an invertible function $h\in\H$ such that $\sup_{\hat E} |F-h|<\e/2$. 
     Hence, by noticing that $|q|=1$ on $M(\H+C)$, we have $|qh-f|= |qh-qF|<\e/2$ on $E$.
      Therefore, on $E$, 
     $$||\ov z^n q h- \phi||_\infty\leq ||\ov z^n(qh-  f ) ||_\infty+ ||\phi-\ov z^n f||_\infty<\e.$$
     Since $\ov z^n q h$ does not vanish on $M(\H+C)$, that function is invertible in $\H+C$.
     \end{proof}
   
  \parindent=0pt
  {\bf Remark}
   Using Theorem \ref{ranks}  and Theorem \ref{dsr=1} we get a second proof of
   the fact that $\text{bsr}(\H+C)=1$ (see Theorem \ref{bsr=1}).


\begin{thebibliography}{99}

\bibitem{ax} S. Axler, Factorization of $L^\infty$ functions, Ann. Math. 106 (1977), 567--572.

\bibitem{c} L. Carleson, Interpolations by bounded analytic functions and the corona problem, Ann. of Math. (2), 76 (1962), 547--559.

\bibitem{cl} G. Corach, A. Larotonda,  Stable range in Banach algebras, J. Pure and Appl. Algebra
32 (1984), 289--300.

\bibitem{cs1} G. Corach, F. D. Su\'arez,  Stable rank in holomorphic function algebras,
Illinois J. Math. 29 (1985), 627--639.

\bibitem{cs2} G. Corach, F. D. Su\'arez,   On the stable range of uniform algebras and $\H$,
Proc. Amer. Math. Soc.  98 (1986), 607--610. 

\bibitem{cs3} G. Corach, F. D. Su\'arez,  Dense morphisms in commutative Banach algebras,
Trans. Amer. Math. Soc. 304 (1987), 537--547.

\bibitem{ga} J. B. Garnett,  Bounded Analytic Functions, 
Pure and Applied Mathematics, 96, Academic Press, New York-London, 1981. xvi+467

 
\bibitem{gm} P. Gorkin, R. Mortini,  Interpolating Blaschke products and factorization in Douglas algebras,  Michigan Math. J. \textbf{38} (1991), 147--160.

 
\bibitem{hof1} K. Hoffman,  Bounded analytic functions and Gleason parts, Ann. of Math. (2) 86 (1967), 74--111.

\bibitem{jmw} P. W. Jones, D. Marshall, T. Wolff, Stable rank of the disc algebra, Proc. Amer. Math. Soc. {96} (1986),  603--604.
  
\bibitem{la} L. Laroco, Stable rank and approximation theorems in $\H$, 
Trans. Amer. Math. Soc.  327 (1991), 815--832.

 
\bibitem{mw} R. Mortini, B. D. Wick, The Bass and topological stable ranks of $H^\infty_\R(\D)$ and $A_\R(\D)$, to appear in J. Reine  Angew. Math. (Crelle Journal).

\bibitem{nisu} A. Nicolau, D. Su\'arez, Approximation by invertible functions in $\H$, Math. Scand. 99 (2006), 287--319.

\bibitem{nik} N. K. Nikolski, Treatise on the Shift Operator, Springer-Verlag, Berlin, 1986.

\bibitem{ni} N. K. Nikolski, In search of the invisible spectrum, Ann. Inst. Fourier (Grenoble) 49 (1999), no. 6, 1925--1998.



\bibitem{ri} M. Rieffel, Dimension and stable rank in the $K$-theory of $C^\ast$-algebras, Proc. London Math. Soc. 46 (1983), 301--333.

\bibitem{ru1} R. Rupp,  Stable ranks of subalgebras of the disc algebra,  
Proc. Amer. Math. Soc.  108 (1990), 137--142.

\bibitem{ru2} R. Rupp,  Stable rank of holomorphic function algebras,
Studia Math. 97 (1990), 85--90.


\bibitem{ru3} R. Rupp,  Stable rank and the $\overline\partial$-equation,
Canad. Math. Bull. 34 (1991), 113--118.

\bibitem{ru4} R. Rupp, Stable rank and boundary principle, 
Topology Appl. 40 (1991), 307--316. 

\bibitem{ro} M. Rosenblum, A corona theorem for countably many functions,
Integral Equations Operator Theory 3 (1980), 125--137.

\bibitem{sa} D. Sarason, Algebras of functions  on the unit circle, Bull. Amer. Math. Soc.
79 (1973), 286--299.

\bibitem{su0} D. Su\'arez, Cech cohomology and covering dimension for the $\H$ maximal ideal space, J. Funct. Anal. 123 (1994), 233--263.

\bibitem{su} D. Su\'arez, Trivial Gleason parts and the topological stable rank of $\H$,
Amer. J. Math. 118 (1996), 879--904.

\bibitem{to} V. Tolokonnikov, Estimates in the Carleson corona theorem, 
 Ideals of the algebra $\H$, a Problem of S.-Nagy, J. Soviet Math. 22 (1983), 1814--1828.
 
 \bibitem{tr} S. Treil, The stable rank of $\H$ equals 1, J. Funct. Anal. 109 (1992), 130--154.
 
\bibitem{tr2} S. Treil,   Estimates in the corona theorem and ideals of $\H$:
a problem of T. Wolff, J. d'Analyse Math.  87 (2002), 481--495.

\bibitem{tw} S. Treil, B. D. Wick, The matrix-valued $H^p$ corona problem
in the disk and polydisk, J. Funct. Anal. 226 (2005), 138--172.

\bibitem{trent} T. T. Trent, A new estimate for the vector valued corona problem,  J. Funct. Anal. 189 (2002), 267--282.

 \bibitem{va}
L. Vasershtein,
Stable rank of rings and dimensionality of topological spaces,
Funct. Anal. Appl. {\bf 5} (1971), 102--110;
translation from Funkts. Anal. Prilozh. {\bf 5} (1971), No.2, 17-27.



\end{thebibliography}
 \end{document}